\documentclass[12pt]{amsart}
\usepackage{amssymb}
\usepackage{amsmath}
\usepackage{latexsym}
\usepackage{amsthm}
\usepackage{graphics}
\usepackage{mathrsfs}
\usepackage[utf8]{inputenc}
\usepackage{tikz}

\usetikzlibrary{positioning}
\usetikzlibrary{arrows}

 \makeatletter
\def\thmhead@plain#1#2#3{%
 \thmname{#1}\thmnumber{\@ifnotempty{#1}{
 }#2}%
 \thmnote{ \the\thm@notefont(#3)}}
\let\thmhead\thmhead@plain
\def\swappedhead#1#2#3{%
 \thmnumber{#2}\thmname{\@ifnotempty{#2}{. }#1}%
 \thmnote{ \the\thm@notefont(#3)}}
\makeatother

\theoremstyle{definition} 
\newtheorem{definition}{Definition}[section]
\newtheorem{remark}[definition]{Remark}

\theoremstyle{plain}      

\newtheorem{theorem}[definition]{Theorem}

\newtheorem{lemma}[definition]{Lemma}

\begin{document}

\keywords{de Bruijn Sequence, generalized de Bruijn graph,
Hamiltonian cycle, universal cycle, nonlinear feedback shift
register, generation of combinatorial objects}
\title[]    {Hamiltonicity of the Cross-Join Graph of de~ Bruijn Sequences}
\author[A. Alhakim]{Abbas Alhakim\\Department of Mathematics\\
American University of Beirut\\
Beirut, Lebanon\\
Email: aa145@aub.edu.lb}
\begin{abstract}
A generalized de Bruijn digraph generalizes a de Bruijn digraph to
the case where the number of vertices need not be a pure power of an
integer. Hamiltonian cycles in these digraphs thus generalize
regular de~Bruijn cycles, and we will thus refer to them simply as
de Bruijn cycles. We define the cross-join to be the graph with all
de Bruijn cycles as vertices, there is an edge between two of these
vertices if one can be obtained from the other via a cross-join
operation. We show that the cross-join graph is connected. This in
particular means that any regular de Bruijn cycle can be
cross-joined repeatedly to reach any other de Bruijn cycle,
generalizing a result about regular binary de Bruijn cycles by
Mykkeltveit and Szmidt \cite{MykkSzmidt2014}. Furthermore, we
present an algorithm that produces a Hamiltonian path across the
cross-join graph, one that we may call a de~Bruijn sequence of de
Bruijn sequences.
\end{abstract}

\maketitle
\section{Introduction}
Consider an alphabet $A$ of size $d$. A de Bruijn cycle of order $k$
over the alphabet $A$ is a periodic sequence of characters of $A$
such that, within one period, every possible string of size $k$
occurs exactly once as a substring. For example $00110$ and
$0022120110$ are de Bruijn sequences of order $2$ over the alphabets
$\{0,1\}$ and $\{0,1,2\}$ respectively. These sequences were
popularized by de Bruijn \cite{deBruijn46} and Good \cite{Good46}
even though their existence was established much earlier, see
Flye-Sainte Marie \cite{Marie} and Martin \cite{Martin34}. Beyond
the mere proof of existence for any $d$ and $k$, de Bruijn
\cite{deBruijn46} and Good \cite{Good46} established that the number
of de Bruijn cycles is $(d!)^{d^{k-1}}/d^k$.

de Bruijn cycles play a pivotal role in coding theory and
cryptography as they are the main building blocks of many stream
cyphers. The binary case is especially useful, although there has
been much research on sequences with non-binary alphabets. Linear
feedback shift register sequences of maximal length are indeed de
Bruijn sequences that lack the all zero string, but they are known
not to be safe and hence not useful for cryptographic applications,
see Massey \cite{Massey67}, for example. de Bruijn cycles based on
nonlinear feedback functions are far more useful and more numerous
than linear ones. They are, however, far from being mathematically
understood as a class. There are many references in the literature
that introduce and study properties of subclasses of de Bruijn
cycles that are nonlinearly constructed. Golomb initiated this study
in his pioneering work \cite{Golomb67}, but there are newer methods
that have been introduced ever since Golomb first published his book
in 1967. Such methods are typically combinatorial or graph-theoretic
producing one de Bruijn sequences or a collection of similar
sequences. Fredricksen \cite{fred1982} is an excellent review
articles that outlines, and often details, many of these known
methods and algorithms. Much more research on nonlinear sequences
has been done in the past two decades. Some recent publications are
\cite{Dubrova}, \cite{Szmidt}, \cite{Turan}. Another breed of
nonlinear sequences are those produced by efficient successor rules
for both binary and non-binary alphabets as in \cite{sawada1} and
\cite{sawada2}.

One main method of generating de Bruijn sequences is the cross-join
technique, which starts with a de Bruijn sequence and interchanges
two appropriately chosen pairs of vertices to obtain another de
Bruijn sequence. The aim of this paper is to show that any de Bruijn
sequence of any alphabet size can be transformed via a sequence of
cross-joins into any other de Bruijn sequence of the same order. In
the binary case, this result was recently proven by Mikkeiltveit and
Szmidt~\cite{MykkSzmidt2014}. Our method applies in the binary case
as well. Our main objective is to investigate cross-join
connectivity of non-binary de Bruijn sequences. However, it turns
out that our method of proof applies without modification to a
generalized version of de Bruijn sequences where the number of
vertices need not be a pure power of some alphabet size $d$.  We
give our proof in the language of the latter, in order to widen the
scope of the result as well as to pin-point to the actual assumption
that is sufficient for the proof. We explain the set up of
generalized de Bruijn digraphs in Section~\ref{Sec:gendB}. In
Section~\ref{S:connectedness} we formulate and prove a fundamental
lemma (Lemma~\ref{L:fundamental}) and use it to prove connectedness
of the cross-join graph (Definition~\ref{D:CR_graph}). In
Section~\ref{S:hamiltonicity} we present an algorithm that generates
a Hamiltonian path of de~Bruijn cycles, given an arbitrary de Bruijn
cycle. We also use the fundamental lemma to establish the
correctness of this algorithm. We round up in
Section~\ref{S:conclusion} with some open problems. An appendix
gives a working implementation of the algorithm in the open source
language R. In the next section we give some background material.

\section{Preliminaries}
There are several known methods to generate de Bruijn sequences. The
one method that is most understood is the algebraic method that we
outline next. First we assume that the alphabet $A$ is residue ring
$\mathbb{Z}_d$. A construction is achieved if we have an initial
state $(s_1,\ldots,s_k)$ and a rule that provides the next symbol
$a_{k+1}$ of the alphabet $A$ when the current state is
$(a_1,\ldots, a_k)$. In this case the next state is
$(a_2,\ldots,a_k,a_{k+1})$. Thus, the next state is determined by a
recurrence relation of order $k$ and of the form
$$a_{i}=f(a_{i-k},\ldots,a_{i-1})$$
where $i\geq k+1$ and \textit{the feedback function} $f$ maps $A^k$
to $A$. Since the number of possible states is finite, iterating the feedback function gives a periodic sequence. de Bruijn sequences
correspond to recurrence functions with the maximal period of $d^k$.
The sequence is called linear when the recurrence function is linear
and homogeneous. When a linear homogeneous feedback function is
applied to an initial string of zeros we obtain the constant zero
sequence. Therefore, the maximal possible period of a linear
feedback function is $d^k-1$. A sequence constructed this way is one
that misses the all-zero string of size $k$. Such a sequence is
called a maximal length linear feedback shift register sequence (LFSR), or
sometimes a ``punctured'' linear de Bruijn sequence, because a full
de Bruijn sequence can then be obtained by appending a zero next to
any of the $d-1$ occurrences of $k-1$ consecutive zeros.

We will concern ourselves here with a method called the cross-join
method, which starts with a de Bruijn cycle of some order and
produces another de Bruijn cycle of the same order.  To see how this
is done, let $w_k=(x_1,\ldots,x_k)$ be a word of size $k$. Any word
$(\hat{x}_1,x_2,\ldots,x_k)$ with $\hat{x}_1\neq x_1$ is a conjugate
of $w_k$. Note that two conjugate words can reach the exact same
words $(x_2,\ldots,x_k,a)$ for $a\in A$. It can be seen now that
interchanging the successors of two conjugate words in a de Bruijn
cycle splits it into two smaller cycles. These two cycles can be
joined back into another de Bruijn cylcle by locating a word of one
cycle whose conjugate is on the other cycle, and interchanging the
successors of these two words. For example, interchanging the
successors of $110$ and $010$ in the sequence $0001110100$ we get the two cycle $00011100$ and
$0101$ (which can be thought of as the sequences of vertices
$(000,001,011,111,110,100,000)$ and  $(010,101,010)$.  Since the
conjugate words $001$ and $101$ are not on the same cycle,
interchaninging their succesors joins the two cycles into the de
Bruijn cycle (arranged to start at $000$): $0001011100$.

Two pairs of vertices that allow to transform a de Bruijn cycle $u$
to another de Bruijn cycle $v$ are called cross-join pairs. In the
binary case, each vertex has exactly one conjugate, so it is enough
to determine one vertex from each pair and the resulting pair is
called a cross-join pair. The idea of cross-joining a binary de
Bruijn sequence into another has been attractive to many
investigators, especially because it is always possible to generate
a de Bruijn sequence via LFSR logic and cross-join it one or more
times to obtain another nonlinear de Bruijn sequence. Chang et. al.
\cite{Chang} conjectured that any two binary de Bruijn sequences
obtained by maximal period LFSR sequences of the same order $k$ have
the same number of cross-join pairs that only depends on $k$, and
therefore the same number of de Bruijn sequences can be made from
each LFSR by using a single cross-join operation. This common number
is $\displaystyle\frac{(2^{k-1}-1)(2^{k-1}-2)}{6}$. The conjecture
was proven by Helleseth and Kl\"{o}ve \cite{hellesethKlove}. The
author of this paper was not able to find any information in the
literature on the number of direct cross-join neighbors of a de
Bruijn sequence when it is not based on a maximal period LFSR. Using
computation to inspect this number of neighbors for low order binary
de Bruijn sequences, we find that many distinct values exist.
Indeed, the $16$ binary de Bruijn sequences of order $4$ are equally
divided between sequences with $7$ cross-join neighbors and
sequences with $10$ cross-join neighbors. The case of binary
sequences of order $5$ is quite different. Table~\ref{T:neighbors}
reports the possible number of neighbors along with the frequency of
vertices that have this number of neighbors. Notice that the formula
of Chang \textit{et al} \cite{Chang} is 7 and 35 respectively for
$k=4$ and $5$, while the corresponding numbers of maximal LFSR
sequences are $2$ and $15$ so that many nonlinear de Bruijn
sequences share this number of neighbors with the LFSR sequences.

\begin{table}
\small
\begin{tabular}{l|lllllllllllllllll}
$n$ & 31 & 32 & 33 & 34 & 35 & 36 & 37 & 38 & 39 & 40 & 41 & 42 & 43
& 44
& 45 & 46 & 47\\

$f$ & 88 & 152 & 240 & 272 & 216 & 136 & 208 & 16 & 176 & 64 & 48 &
16 & 40 & 0 & 112 & 0 & 136\\\hline

$n$ & 48 & 49 & 50 & 51 & 52 & 53 & 54 & 55 & 56 & 57 & 58 & 59 & 60
& 61
& 62 & 63 & 64\\

$f$ & 48 & 32 & 0 & 16 & 0 & 0 & 0 & 0 & 0 & 0 & 0 & 0 & 8 & 0 & 16
& 0 & 8
\end{tabular}
\caption{The frequency $f$ of binary de Bruijn sequences of order
$5$ with $n$ cross-join neighbors for all possible values of $n$}\label{T:neighbors}
\end{table}

Recently, Mykkeltveit and Szmidt \cite{MykkSzmidt2014} settled the
following question in the affirmative. ``Is it possible to obtain
any binary de Bruijn sequence by applying a sequence of cross-join
pair operations to a given binary de Bruijn sequence?'' They
mentioned that this  is a several-decade-old question that was
recently asked at the International Workshop on Coding and
Cryptography 2013 in Bergen. They even claim this result as an
explanation of the origins of nonlinear Boolean functions that yield
de Bruijn cycles. This is based on the fact that once a cross-join
pair is identified, the feedback function of the initial de Bruijn
sequence can easily be altered to give the feedback function of the
new sequence.

As mentioned in the introduction, the main theme of this paper is to
examine the main result of \cite{MykkSzmidt2014} in a more general
setting. In the binary case, it is well known that for a feedback
function $f$ to produce a pure cycle (including a de~Bruijn cycle),
it is necessary that $f(x_1,\ldots,x_k)=x_1+g(x_2,\ldots,x_k)$. It
is the function $g$ that Mykkeltveit and Szmidt
\cite{MykkSzmidt2014} used in their proof. Unfortunately, there is
no such representation for non-binary feedback functions, calling
for a different method of proof.

We formulate and prove our results for the set of Hamiltonian cycles
(the analogues of de Bruijn cycles) in the so-called $(d,N)$
generalized de Bruijn digraphs, in which the number of vertices is
any positive integer $N\geq2$ and which boils down to a regular de
Bruijn dighraph of order $k$ when $N=d^k$.

\section{Generalized de Bruijn digraphs}\label{Sec:gendB}
A de Bruijn sequence is interchangeably called a de Bruijn cycle
because it can be seen as a Hamiltonian cycle on the de Bruijn
digraph.  A de Bruijn digraph with alphabet $A=\{0,1,\ldots,d-1\}$
has $d^k$ vertices which can be taken as the set of $d^k$ vectors
$(x_1,\ldots,x_k)$. For two vertices $\textbf{a}=(a_1,\ldots,a_k)$
and $\textbf{b}=(b_1,\ldots,b_k)$, there exists an edge connecting
$\textbf{a}$ to $\textbf{b}$ if and only if $b_i=a_{i+1}$ for
$i=1,\ldots,k-1$. When the vertices are regarded as decimal numbers
represented in base $d$, the set of edges consists of all pairs
$(x,y)$ (or interchangeably $x\rightarrow y$) where $x$ and $y$ are
integers in $\{0,\ldots,d^k-1\}$ and $y=dx+r\mod d^k$,
$r=0,1\ldots,d-1$.

In a generalized de Bruijn digraph, $d^k$ is replaced by any integer
$N>d$. This digraph was introduced to dispense with the restrictive
number of vertices in ordinary de Bruijn digraphs. Formally, a
generalized de Bruijn digraph $G_B(N,d)=(V,E)$ where the vertex set
is $V=\{0,1,\ldots,N-1\}$ and $(x,y)$ is contained in the edge set
$E$ if and only if $y=dx+r \mod N$ for some $r\in\{0,\ldots,d-1\}$.

This digraph preserves many of the properties of ordinary de Bruijn
digraphs. As shown below, it is a regular digraph. Also, Imase and
Itoh \cite{ImaseItoh}, Reddy, Pradhan and Kuhl \cite{ReddyPK} prove
that $G_B(N,d)$ has a very short diameter just like de Bruijn
digraphs. They also show that  $G_B(N,d)$ is strongly connected.
Indeed, Du and Hwang \cite{DuHwang} show that  $G_B(N,d)$ is
Hamiltonian when $gcd(d,N)>1$. Du et al \cite{Du91} establishes the
result when $gcd(d,N)=1$ and $d>2$. That is, $G_B(N,d)$ is
Hamiltonian except when $d=2$ and odd $N$. In the rest of the paper
the following notation will be used. If $(x,y)\in E$, we say that
$y$ is a successor of $x$ and that $x$ is a predecessor of $y$. The
set of possible successors of a vertex $x$ is denoted by
$\Gamma_x^+$ while the set of predecessors is $\Gamma_x^-$. Two
vertices $x_1$ and $x_2$ are said to be conjugate vertices if there
exist two vertices $y_1$ and $y_2$ such that $(x_1,y_1)$,
$(x_1,y_2)$, $(x_2,y_1)$ and $(x_2,y_2)$ are all edges in
$G_B(N,d)$. In this case we also say that $y_1$ and $y_2$ are
companion vertices. A path is a sequence of vertices $x_1,\ldots,
x_k$ such that $(x_i,x_{i+1})$ is an edge for all $i=1,\ldots,k-1$.
A path is simple if all of its vertices are distinct. A cycle is a
path in which the first and last vertices coincide. A cycle is
simple if, except for the first and last, all of its vertices are
distinct. A Hamiltonian path (resp. cycle) is a simple path (resp.
cycle) that includes all the vertices of the digraph. Since
$G_B(N,d)$ generalizes de Bruijn digraphs, we are going to refer to
a Hamiltonian cycle of $G_B(N,d)$ simply as a de~Bruijn cycle. When
$N=d^k$ for some integers $d>1$ and $k\geq1$, $G_B(N,d)$ reduces to
the regular $d$-ary de~Bruijn digraph of order $k$ which we denote
by $B(d,k)$. Figures~(1) and (2) illustrate regular and generalized
de Bruijn digraphs.

\begin{figure}\label{F:RegulardBGraphs}
\begin{center}
\begin{tabular}{lll}
\begin{tikzpicture}
[->,>=stealth',shorten >=1pt,node distance=1.5cm,auto]
\tikzstyle{main node}=[circle, draw, align=center];

\node[main node] (0) {$000$}; \node[main node] (1) [below left of=0]
{$001$}; \node[main node] (2) [below right of=1] {$010$}; \node[main
node] (4) [below right of=0] {$100$}; \node[main node] (5) [below
of=2] {$101$}; \node[main node] (3) [below left of=5] {$011$};
\node[main node] (6) [below right of=5] {$110$}; \node[main node]
(7) [below right of=3] {$111$};

\path (0) edge [in=120,out=60,loop] node {} (0)
    edge node {} (1)
(1) edge node {} (2)
    edge node {} (3)
(2) edge node {} (4)
    edge [bend right=10] node {} (5)
(3) edge node {} (6)
    edge node {} (7)
(4) edge node {} (0)
    edge node {} (1)
(5) edge [bend right=10] node {} (2)
    edge node {} (3)
(6) edge node {} (4)
    edge node {} (5)
(7) edge node {} (6)
    edge [in=300,out=240,loop] node {} (7);
\end{tikzpicture}
&\;\;\;\;\;
\begin{tikzpicture}[->,>=stealth',shorten >=1pt,node distance=1.5cm,auto]
\tikzstyle{main node}=[circle, draw, align=center];

\node[main node] (0) {$00$};

\node[below of=0] (v0) {};

\node[left of=v0] (v1) {};

\node[right of=v0] (v2) {};

\node[main node, left of=v1] (1) {$02$};

\node[main node, draw, circle, right of=v2] (2) {$20$};

\node[below of=1] (v3) {};

\node[main node, right of=v3] (3) {$10$};

\node[right of=3] (v4) {};

\node[main node, right of=v4] (4) {$01$};

\node[right of=4] (v5) {};

\node[main node, right of=v5] (5) {$22$};

\node[main node, below of=v3] (6) {$21$};

\node[right of=6] (v6) {};

\node[right of=v6] (v7) {};

\node[right of=v7] (v8) {};

\node[main node, right of=v8] (7) {$12$};

\node[main node, below of=v7] (8) {$11$};

\path (0) edge [in=120,out=60,loop] node {} (0)
    edge [bend right=10] node {} (1)
    edge node {} (4)
(1) edge [bend right=10] node {} (6)
    edge [bend left=5] node {} (2)
    edge node {} (5)
(2) edge [bend right=10] node {} (0)
    edge [bend left=5] node {} (1)
    edge [bend right=10] node {} (4)
(3) edge [bend right=10] node {} (1)
    edge [bend left=20] node {} (4)
    edge  node {} (0)
(4) edge [bend left=20] node {} (3)
    edge [bend right=10] node {} (7)
    edge  node {} (8)
(5) edge [in=30,out=-30,loop] node {} (5)
    edge node {} (2)
    edge node {} (6)
(6) edge [bend right=10] node {} (3)
    edge [bend right=10] node {} (8)
    edge [bend right=5] node {} (7)
(7) edge [bend right=10] node {} (2)
    edge node {} (5)
    edge [bend right=5] node {} (6)
(8) edge [in=-60,out=-120,loop] node {} (8)
    edge [bend right=10] node {} (7)
    edge node {} (3);
\end{tikzpicture}
\end{tabular}
\caption{de Bruijn digraphs $B(3,2)$ (left) and $B(2,3)$ (right).}
\end{center}
\end{figure}
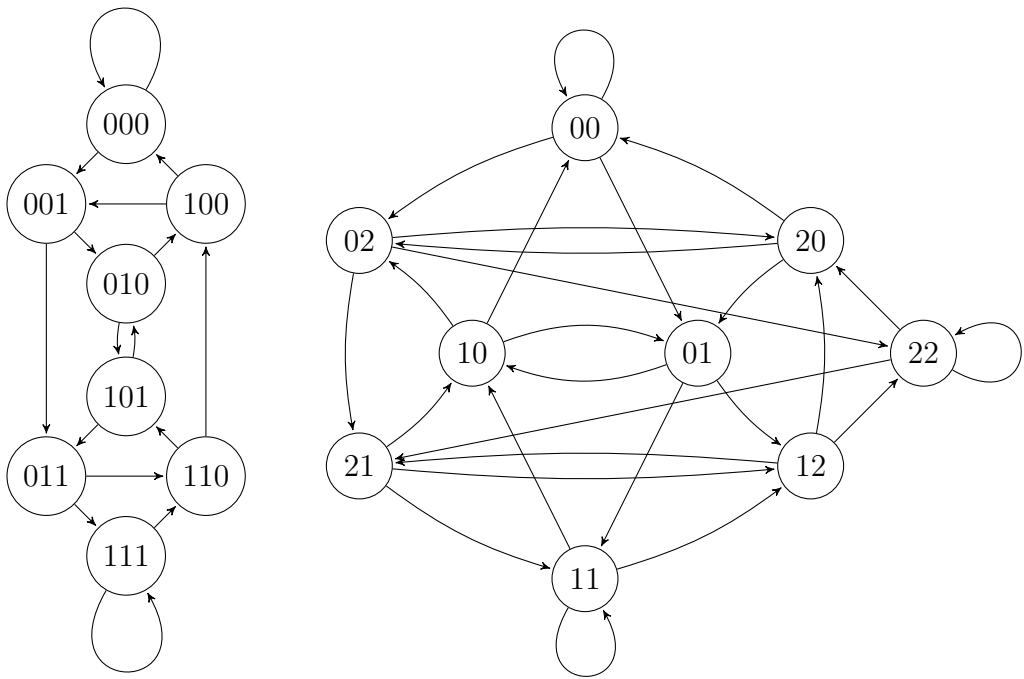

A cross-join operation is performed to a de~Bruijn cycle $dB$ as
follows. First locate a conjugate pair of vertices on $dB$ and
interchange their successors, thus splitting $dB$ into two shorter
cycles. Next locate another conjugate pair of vertices, each
residing on a different cycle, and interchange their successors to
obtain a new de~Bruijn cycle. We illustrate with two examples. The
first one uses the de~Bruijn cycle $0,1,3,7,6,5,2,4,0$ of $G_B(8,2)$
(with a $0$ at the end emphasizing the cyclic nature). Interchanging
the successors of $1$ and $5$ (the cross pair) we get the two cycles
$0,1,2,4,0$ and $3,7,6,5,3$. Interchanging the successors of the
conjugate pair $(2,6)$ (the join pair) we get the de~Bruijn sequence
$0,1,2,5,3,7,6,4,0$.

The reader is urged to verify that the de~Bruijn cycle
$0,2,1,5,3,4,0$ of $G_B(6,3)$ can be cross-joined into
$0,1,5,3,4,2,0$ using the two pairs $(2,4)$ and $(0,4)$
consecutively.

\begin{definition}\label{D:adjacent_cycles}
Two de Bruijn cycles are called cross-join adjacent, or simply adjacent, if it is possible to
cross-join one into another via a single cross-join operation.
\end{definition}
This allows for the following definition.
\begin{definition}\label{D:CR_graph}
The cross-join graph $\mathcal{C}(N,d)$ of $G_B(N,d)$ has the set of
all de~Bruijn cycles as the set of vertices. There is an edge
between two vertices if they are adjacent in the sense of
definition~\ref{D:adjacent_cycles}. In particular,
$\mathcal{C}(d,d^k)$ is the cross-join graph of a regular de Bruijn
digraph of order $k$.
\end{definition}

\begin{lemma}\label{L:regular}
$G_B(N,d)$ is $d$-regular. That is, each vertex has $d$ successors
and $d$ predecessors.
\end{lemma}

\begin{proof}
We only need to show that each vertex in $G_B(N,d)$ has exactly $d$
predecessors. Let $\gcd(d,N)=\delta$. Given a vertex $y$ between $0$
and $N-1$, we need to count the number of solutions to the equations
\[
y\equiv dx+r \mod N,
\]
where $r$ can be $0,1,\ldots,d-1$. Equivalently, $dx\equiv(y-r)\mod
N$. When $y-r$ is not a multiple of $\delta$ there obviously is no
solution. Suppose that $y-r=\delta t$ for some $t$ in
$\{0,1,\ldots,N/{\delta}-1\}$. The equation implies that
$(d/\delta)x\equiv t\mod N/\delta$. Since
$\gcd(d/\delta,N/\delta)=1$, the last equation admits a unique
solution modulo $N/\delta$ and therefore it has $\delta$ solutions
modulo $N$. The lemma now follows, as there are exactly $d/\delta$
multiples of $\delta$ in the $d$ consecutive integers
$y,y-1,\ldots,y-d+1$ for each fixed $y$.
\end{proof}

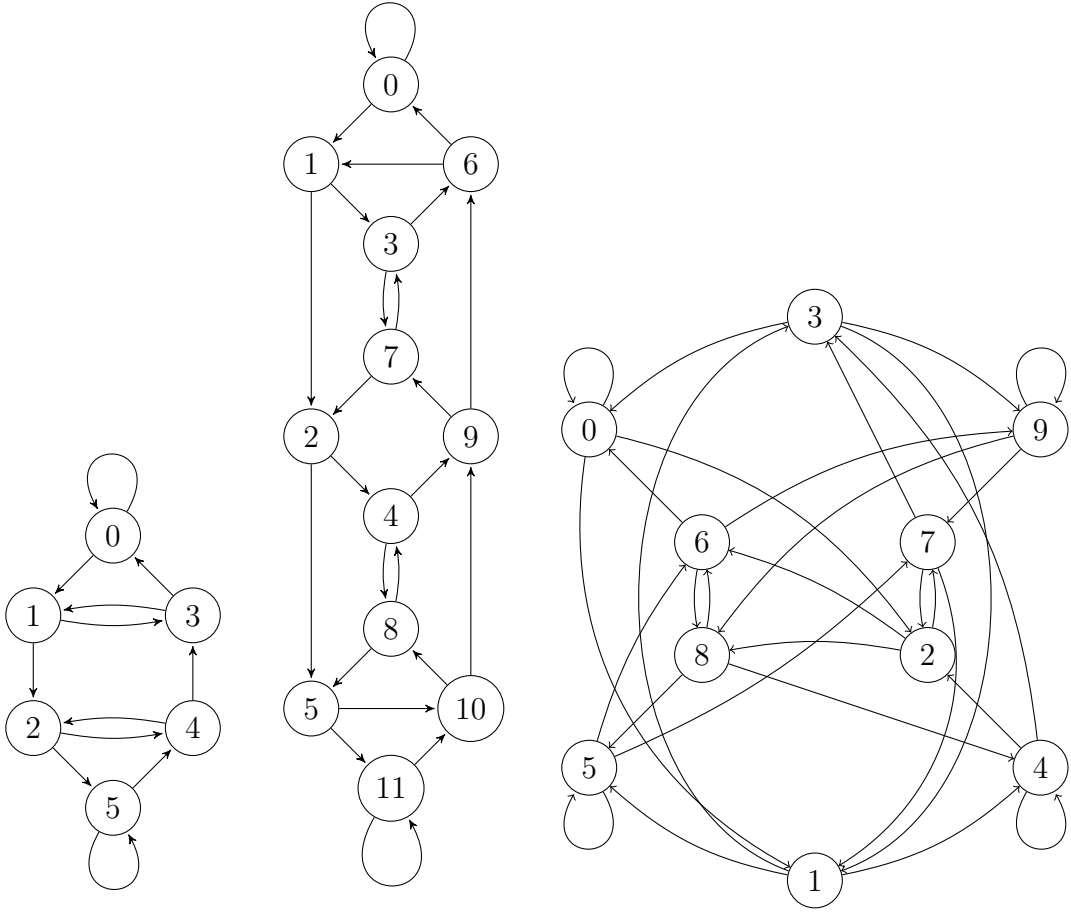
\begin{figure}[hbtp]\label{F:gendBGraphs}
\begin{center}
\begin{tabular}{lll}
\begin{tikzpicture}
[->,>=stealth',shorten >=1pt,node distance=1.5cm,auto]
\tikzstyle{main node}=[circle, draw, align=center];

\node[main node] (0) {$0$}; \node[main node] (1) [below left of=0]
{$1$}; \node[main node] (3) [below right of=0] {$3$}; \node[main
node] (2) [below of=1] {$2$}; \node[main node] (4) [below of=3]
{$4$}; \node[main node] (5) [below right of=2] {$5$};

\path (0) edge [in=120,out=60,loop] node {} (0)
    edge node {} (1)
(1) edge node {} (2)
    edge [bend right=10] node {} (3)
(2) edge node {} (5)
    edge [bend right=10] node {} (4)
(3) edge node {} (0)
    edge [bend right=10] node {} (1)
(4) edge node {} (3)
    edge [bend right=10] node {} (2)
(5) edge node {} (4)
    edge [in=300,out=240,loop] node {} (5);
\end{tikzpicture}
&\;\;\;
\begin{tikzpicture}
[->,>=stealth',shorten >=1pt,node distance=1.5cm,auto]
\tikzstyle{main node}=[circle, draw, align=center];

\node[main node] (0) {$0$}; \node[main node] (1) [below left of=0]
{$1$}; \node[main node] (6) [below right of=0] {$6$}; \node[main
node] (3) [below right of=1] {$3$}; \node[main node] (7) [below
of=3] {$7$}; \node[main node] (2) [below left of=7] {$2$};
\node[main node] (9) [below right of=7] {$9$}; \node[main node] (4)
[below right of=2] {$4$}; \node[main node] (8) [below of=4] {$8$};
\node[main node] (5) [below left of=8] {$5$}; \node[main node] (10)
[below right of=8] {$10$}; \node[main node] (11) [below right of=5]
{$11$};

\path (0) edge [in=120,out=60,loop] node {} (0)
    edge node {} (1)
(1) edge node {} (2)
    edge node {} (3)
(2) edge node {} (4)
    edge node {} (5)
(3) edge node {} (6)
    edge [bend right=10] node {} (7)
(4) edge [bend right=10] node {} (8)
    edge node {} (9)
(5) edge node {} (10)
    edge node {} (11)
(6) edge node {} (0)
    edge node {} (1)
(7) edge node {} (2)
    edge [bend right=10] node {} (3)
(8) edge [bend right=10] node {} (4)
    edge node {}  (5)
(9) edge node {} (6)
    edge node {} (7)
(10) edge node {}  (8)
     edge node {}  (9)
(11) edge node {} (10)
     edge [in=300,out=240,loop] node {} (11);
\end{tikzpicture}
&
\begin{tikzpicture}
    \node[shape=circle,draw=black] (3) at (0,0) {3};
    \node[shape=circle,draw=black] (0) at (-3,-1.5) {0};
    \node[shape=circle,draw=black] (9) at (3,-1.5) {9};
    \node[shape=circle,draw=black] (6) at (-1.5,-3) {6};
    \node[shape=circle,draw=black] (7) at (1.5,-3) {7};
    \node[shape=circle,draw=black] (8) at (-1.5,-4.5) {8};
    \node[shape=circle,draw=black] (2) at (1.5,-4.5) {2};
    \node[shape=circle,draw=black] (5) at (-3,-6) {5};
    \node[shape=circle,draw=black] (4) at (3,-6) {4};
    \node[shape=circle,draw=black] (1) at (0,-7.5) {1};

    \path [->] (3) edge [bend right=15] node {}  (0)
    edge [bend left=15] node {}  (9)
    edge [bend left=70] node {} (1);
    \path [->] (0) edge [in=120,out=60,loop] node {} (0)
                    edge [bend right=35] node {} (1)
                    edge [bend left=20] node {} (2);
    \path [->] (9) edge [in=60,out=120,loop] node {} (9)
                    edge node {} (7)
                    edge [bend right=20] node {} (8);
    \path [->] (6) edge [bend right=10] node {} (8)
                    edge node {} (0)
                    edge [bend left=15] node {} (9);
     \path [->] (7) edge [bend right=10] node {} (2)
                    edge node {} (3)
                    edge [bend left=40] node {} (1);
     \path [->] (8) edge [bend right=10] node {} (6)
                    edge node {} (4)
                    edge  node {} (5);
     \path [->] (2) edge [bend right=10] node {} (7)
                    edge [bend right=10] node {} (6)
                    edge [bend right=10] node {} (8);
     \path [->] (5) edge [in=240,out=300,loop] node {} (5)
                    edge [bend right=10] node {} (7)
                    edge [bend left=10] node {} (6);
     \path [->] (4) edge [in=300,out=240,loop] node {} (4)
                    edge node {} (2)
                    edge [bend right=20] node {} (3);
     \path [->] (1) edge [bend left=70] node {} (3)
                    edge [bend left=15] node {} (5)
                    edge [bend right=15] node {} (4);
    \end{tikzpicture}
\end{tabular}
\caption{generalized de Bruijn digraphs $G_B(2,6)$ (left),
$G_B(2,12)$ (middle), and $G_B(3,10)$ (right).}
\end{center}
\end{figure}

\begin{lemma}
Suppose that $d$ divides $N$, then for any vertex $y\in V$,
$\Gamma_y^-=\{\lfloor y/d\rfloor+t\cdot N/d: t=0,\ldots,d-1\}$.
\end{lemma}
\begin{proof}
Since $y/d-1<\lfloor y/d\rfloor\leq y/d$, we have
\[
y-d<d\lfloor y/d\rfloor\leq y,
\]
so that $y=d\lfloor y/d\rfloor+r$ for some $r\in\{0,\ldots,d-1\}$.
Hence, $\lfloor y/d\rfloor$ is a predecessor of $y$, and clearly so
is $\lfloor y/d\rfloor+t\cdot N/d$ for each $t=1,\ldots,d-1$. By
lemma~\ref{L:regular}, these are the only predecessors of $y$.
\end{proof}

If follows that, when $d$ divides $N$, for each vertex
$x\in\{0,\ldots,N/d-1\}$ and each $t=1,\ldots,d-1$,
$\Gamma_x^+=\Gamma_{x+\left(N/d\right)\cdot  t}^+$. We formulate
this in the following form that is more usable below.

\begin{lemma}\label{L:Successor}
Suppose that $d$ divides $N$. Let $y_1$ and $y_2$ be two successors
of $x_1$. Suppose that $y_1$ is the successor of some $x_2\neq x_1$.
Then $y_2$ is also a successor of $x_2$.
\end{lemma}
 The reason we require that $d$ divides $N$ in the above two lemmas
 is precisely to require any two conjugate vertices $x_1$ and $x_2$
 to be \textit{completely conjugates}, in the sense that
 $\Gamma_{x_1}^+=\Gamma_{x_2}^+$. In fact, when $d$ does not divide
 $N$, the conjugacy relation is not transitive. That is, it is
 possible to find vertices $x_1,x_2,x_3$ such that $x_1,x_2$ is a
 conjugate pair, $x_2,x_3$ is a conjugate pair but $x_1,x_3$
 is not a conjugate pair. As an illustration, we list in
 Table~\ref{T:edges} all the edges of $G_B(N,d)$ for $d=4$ and three values of $N$
 with different divisibility conditions. The edges are listed in the
 form $x\rightarrow y_1,\ldots,y_{d-1}$, meaning that $(x,y_i)$ is
 an edge for each $i$. In Case (a) note that $0$ and $2$ are
 conjugates (as they both have an edge to $0$ and $1$), $2$ and $4$ are
 conjugates but $0$ and $4$ are not conjugates. Similar intransitive
 vertices can be found in (b). Only in (c), where $d|N$, conjugate
 vertices form an equivalent class and so they are completely
 conjugate.

 \begin{table}
 \begin{tabular}{c|c|c}
 \begin{tabular}{lccccc}
 0 & $\rightarrow$ & 0, & 1, & 2, & 3\\
 1 & $\rightarrow$ & 4, & 5, & 6, & 7\\
 2 & $\rightarrow$ & 8, & 9, & 0, & 1\\
 3 & $\rightarrow$ & 2, & 3, & 4, & 5\\
 4 & $\rightarrow$ & 6, & 7, & 8, & 9\\
 5 & $\rightarrow$ & 0, & 1, & 2, & 3\\
 6 & $\rightarrow$ & 4, & 5, & 6, & 7\\
 7 & $\rightarrow$ & 8, & 9, & 0, & 1\\
 8 & $\rightarrow$ & 2, & 3, & 4, & 5\\
 9 & $\rightarrow$ & 6, & 7, & 8, & 9\\
 \end{tabular}
 &
\begin{tabular}{lccccc}
 0 & $\rightarrow$ & 0, & 1, & 2, & 3\\
 1 & $\rightarrow$ & 4, & 5, & 6, & 7\\
 2 & $\rightarrow$ & 8, & 9, & 10, & 0\\
 3 & $\rightarrow$ & 1, &  2, & 3, & 4\\
 4 & $\rightarrow$ & 5, & 6, & 7, & 8\\
 5 & $\rightarrow$ & 9, & 10, & 0, & 1\\
 6 & $\rightarrow$ & 2, & 3, & 4, & 5\\
 7 & $\rightarrow$ & 6, & 7, & 8, & 9\\
 8 & $\rightarrow$ & 10, & 0, & 1, & 2\\
 9 & $\rightarrow$ & 3,  & 4, & 5, & 6\\
 10 & $\rightarrow$ & 7, & 8, & 9, & 10\\
 \end{tabular}
 &
 \begin{tabular}{lccccc}
 0 & $\rightarrow$ & 0, & 1, & 2, & 3\\
 1 & $\rightarrow$ & 4, & 5, & 6, & 7\\
 2 & $\rightarrow$ & 8, & 9, & 10, & 11\\
 3 & $\rightarrow$ & 0, & 1, & 2, & 3\\
 4 & $\rightarrow$ & 4, & 5, & 6, & 7\\
 5 & $\rightarrow$ & 8, & 9, & 10, & 11\\
 6 & $\rightarrow$ & 0, & 1, & 2, & 3\\
 7 & $\rightarrow$ & 4, & 5, & 6, & 7\\
 8 & $\rightarrow$ & 8, & 9, & 10, & 11\\
 9 & $\rightarrow$ & 0, & 1, & 2, & 3\\
10 & $\rightarrow$ & 4, & 5, & 6, & 7\\
11 & $\rightarrow$ & 8, & 9, & 10, & 11\\
 \end{tabular}\\
 \hline (a) $N=10$, $d=4$ & (b) $N=11$, $d=4$ & (c) $N=12$, $d=4$
 \end{tabular}
\caption{}\label{T:edges}
 \end{table}

\section{Connectedness of $\mathcal{C}(N,d)$}\label{S:connectedness}
In this section we formulate and prove our main result. We will
first define a metric distance between de Bruijn cycles. To this end
let us align de Bruijn cycles as finite sequences that all start
with the same vertex. Without loss of generality we choose this
initial vertex to be $0$. In the following, vertices of the
cross-join graph, i.e., de Bruijn sequences, are denoted by $u$,
$v$, etc. while vertices of $G_B(N,d)$ are denoted by $x$, $y$, etc.

\begin{definition}\label{D:distance}
Let $u$ and $v$ be two vertices of $\mathcal{C}(N,d)$ aligned as in
the previous paragraph. The distance $\mathcal{D}(u,v)$ is defined
as $N-L$, where $L$ is the length of the longest initial path that
is common to $u$ and $v$.
\end{definition}

It is straightforward to verify that $\mathcal{D}(u,v)$ satisfies
the three axioms of a distance. Also, it is essential to keep in
mind that this distance is not related to the distance defined by
the graph adjacency on $\mathcal{C}(N,d)$, (i.e. the number of
vertices in the shortest path across edges of the graph between two
vertices). As an example, for $d=3$ and $N=10$, the following are
two de Bruijn sequences aligned to start at $0$, with the last $0$
repeated to stress that the sequences cycle back to the initial
vertex.

$u=(0,2,7,1,5,6,9,8,4,3,0)$

$v=(0,2,7,1,4,3,9,8,5,6,0)$

The maximum common initial path is $(0,2,7,1)$ so
$\mathcal{D}(u,v)=10-4=6$. The following lemma is fundamental for
the rest of the paper.

\begin{lemma}\label{L:fundamental}
Let $u$ and $v$ be two distinct de Bruijn sequences in $G_B(N,d)$
where $d$ divides $N$. Then there exists a de Bruijn sequence $u_1$
which is a neighbor of $u$ in $\mathcal{C}(N,d)$ such that
$\mathcal{D}(u_1,v)<\mathcal{D}(u,v)$.
\end{lemma}

We prove this lemma after we state and prove our main result.

\begin{theorem}
When $d$ divides $N$, the cross-join graph $\mathcal{C}(N,d)$ is
connected.
\end{theorem}

\begin{proof}
Let $u$ and $v$ be two distinct vertices in $\mathcal{C}(N,d)$. By
Lemma~\ref{L:fundamental} $u$ has a neighbor $u_1$ on
$\mathcal{C}(N,d)$ such that $\mathcal{D}(u_1,v)<\mathcal{D}(u,v)$.
If $u_1=v$ then we are done, otherwise the same argument can be
iterated to get a vertex $u_2$, which is a neighbor of $u_1$, with
$\mathcal{D}(u_2,v)<\mathcal{D}(u_1,v)$. Due to the strict
inequality, and since the number of vertices of $\mathcal{C}(N,d)$
is finite, it is evident that this iterative process must end at $v$
after a finite number $m$, leading to the desired path
$u_0=u,u_1,\ldots,u_m=v$.
\end{proof}

\begin{proof}[Proof of Lemma~\ref{L:fundamental}]
As a road map, we are going to state and prove three claims within
the proof, the main one is Claim~1, while Claim~2 and Claim~3 are
stated and proved within the proof of Claim~1.

Let $M_0$ be the maximal common initial sequence of $u$ and $v$.
That is, suppose that the sequence
\[
M_0: 0=x_1\rightarrow x_2\rightarrow\cdots\rightarrow x_{L_0}
\]
is common to $u$ and $v$ and $L_0$ is maximal. Since $u\neq v$,
$L_0<N$ and so the successors of $x_{L_0}$ in $u$ and $v$ are both
distinct from $0$. Let us refer to these successors respectively as
$x^{(1)}$ and $x_{L_0+1}$. Since $u$ is a de Bruijn sequence, it
contains every vertex of $G_B(N,d)$ so it must contain $x_{L_0+1}$.
The latter is evidently one of the vertices of $\bar{M}_0$, the
complement of $M_0$ in $u$; that is, the sub-path of $u$ that starts
with $x^{(1)}$ and ends with the vertex just before $0$. Let
$^{*}x_0$ be the predecessor of $x_{L_0+1}$ in $u$. Since
$x_{L_0+1}$ belongs to $\bar{M}_0$, the vertex $^{*}x_0$ is either
in $\bar{M}_0$ or it is $x_{L_0}$ itself. But the latter is not
possible because otherwise the common initial path of $u$ and $v$
would extend to $x_{L_0+1}$, defying the maximality of $M_0$. Now
$x_{L_0}$ and $^{*}x_0$ are predecessors of the same vertex. They
must be conjugate vertices by Lemma~\ref{L:Successor}.  Swapping
their successors we split $u$ into two cycles, a cycle $C_1$ that
includes the vertex $0$ and another cycle $\tilde{C}_1$ that
includes the edge $^{*}x_0\rightarrow x^{(1)}$.

The cycle $C_1$, aligned to start at $0$, and the de Bruijn cycle
$v$ have a maximal common initial sequence
\[
M_1: 0=x_1\rightarrow x_2\rightarrow\cdots\rightarrow
x_{L_0}\rightarrow\cdots\rightarrow x_{L_1}
\]
where $L_1\geq L_0+1$. The rest of the proof depends on establishing
the following

Claim 1: It is possible to join $C_1$ and $\tilde{C}_1$ by using
vertices in $\bar{M}_1$, the complement of $M_1$ in $C_1$.

To show this suppose we cannot. Then let the successors of $x_{L_1}$
in $v$ and $C_1$ be $x_{L_1+1}$ and $x^{(2)}$ respectively.
Obviously, $x_{L_1+1}$ is not on the path $M_1$. Since every
possible vertex is either on $C_1$ or on $\tilde{C}_1$, it follows
that $x_{L_1+1}$ is on $\bar{M}_1$, the complement of $M_1$ in
$C_1$, as it cannot be on $\tilde{C}_1$, by our assumption. Let
$^{*}x_1$ be the predecessor of $x_{L_1+1}$ in $C_1$. Similarly to
the previous paragraph, we can argue that $^{*}x_1$ is in
$\bar{M}_1$.

Interchanging the successors of $x_{L_1}$ and $^{*}x_1$ we further
split the cycle $C_1$ into two cycles $C_2$ and $\tilde{C}_2$ with
$C_2$ being the cycle that includes $0$ and shares a larger still
initial path with $v$, say,
\[
M_2: 0=x_1\rightarrow x_2\rightarrow\cdots\rightarrow
x_{L_2},L_2>L_1.
\]
In essence, this process can be iterated, arranging and re-arranging
vertices on the initial cycle $C_1$ but without using vertices from
$\tilde{C}_1$, only a finite number of times. Let $k$ be the maximal
number of iterations and let $C_k$ be the resulting cycle that
includes the vertex $0$ with maximal initial path
\[
M_k: 0=x_1\rightarrow x_2\rightarrow\cdots\rightarrow x_{L_k},
L_k>L_{k-1}
\]
that is common with the de Bruijn sequence $v$.\vspace{8pt}

Claim 2: The sub-path of the cycle $C_k$ that begins with $x_{L_k}$
and ends with $0$ is simply an edge $(x_{L_k},0)$. That is, there is
no vertex in $C_k$ between $x_{L_k}$ and $0$.

To see this, suppose that $x^{(k+1)}\neq0$ is the successor of
$x_{L_k}$ in $C_k$. Let $x_{L_k+1}$ be the successor of $x_{L_k}$ in
$v$, so that $x_{L_k+1}$ and $x^{(k+1)}$ are companion vertices. We
then see that $x_{L_k+1}\neq0$ since otherwise the de Bruijn cycle
$v$ would be shorter than $C_k$. Since $C_1$ and $\tilde{C}_1$
include all vertices, $x_{L_k+1}$ is either in $\tilde{C}_1$ or it
is in the part $\bar{M}_1$ of $C_1$. If the first case is true,
swapping the predecessor of $x_{L_k+1}$ in $\tilde{C}_1$ with the
predecessor of $x^{(k+1)}$ (which is evidently one of the vertices
of $\bar{M}_1$) shows that $C_1$ and $\tilde{C}_1$ can be joined
into a de Bruijn sequence using a vertex outside $M_1$,
contradicting the original assumption of Claim 1.

If the second case is true, that is, if $x_{L_k+1}$ belongs to
$\bar{M}_k$ or any of the cycles made by the previous iteration and
that are at most  $C_2,\ldots,C_{k-1},
\tilde{C}_2,\ldots,\tilde{C}_{k-1}$ (equivalently, if it is one of
the vertices of $\bar{M}_1$), then we can swap the predecessors of
$x^{(k+1)}$ and $x_{L_k+1}$ to get yet another cycle $C_{k+1}$ that
shares a longer initial segment with $v$, contradicting the
maximality of $C_k$.

It follows that $x_{L_k}$ is a predecessor of $0$. We next prove

Claim 3: $C_k$ includes all predecessors of $0$.\vspace{8pt}

We prove this claim in a way similar to the proof of the previous
claim. In effect, suppose that $y$ is a predecessor of $0$ that is
not on $C_k$. If $y$ belongs to $\tilde{C}_1$ we get a contradiction
because we could have joined $C_1$ and $\tilde{C}_1$ by swapping the
successors of $y$ and $x_{L_k}$ (which is on $\bar{M}_1$). Likewise,
the presence of $y$ on any of the intermediate cycles
$C_2,\ldots,C_{k-1}, \tilde{C}_2,\ldots,\tilde{C}_{k-1}$ contradicts
the maximality of $k$.

The validity of this last claim means that the sequence $M_k$ cannot
be continued into a de Bruijn sequence as it cannot cycle back to
$0$ without using one of the predecessors of $0$ for a second time.
This of course is not true because $M_k$ is already an initial path
of the de Bruijn sequence $v$.

We have thus proven that $C_1$ and $\tilde{C}_1$ can be joined by
swapping the successor of a vertex in $\bar{M}_1$ with that of a
conjugate vertex in $\tilde{C}_1$. This makes a new de Bruijn
sequence $u_1$ which is a neighbor of $u$ on $\mathcal{C}(N,d)$.
Since $L_0<L_1$, $N-L_1<N-L_0$ and $u_1$ satisfies the inequality
\[
\mathcal{D}(u_1,v)<\mathcal{D}(u,v)
\]
as desired.
\end{proof}

\section{Hamiltonicity of $\mathcal{C}(N,d)$}\label{S:hamiltonicity}
The next natural question is whether or not the cross-join graph
admits a Hamiltonian path or cycle. There is no simple answer as we
don't have information about the number of cross-join neighbors of
each vertex/sequence, so none of the known sufficient conditions for
Hamiltonicity can be applied. It turns out that the principle upon
which Lemma~\ref{L:fundamental} is based can be used again to show
the existence of a Hamiltonian path by actually generating one.
Namely, that we can cross-join a sequence $u$ into another sequence
by keeping the longest possible initial subsequence of $u$
unchanged. The following algorithm describes the steps.\newline

ALGORITHM H.

Input: A de Bruijn sequence $u_1=(x_1=0,x_2,\ldots,x_N)$

Output: A Hamiltonian path along the cross-join graph
$\mathcal{C}(N,d)$


1. For the current sequence $u_k$, $k\geq1$ form the set

$$A_k=\{(i,i^{\prime}): i<i^{\prime} \text{ and } (x_i,x_{i^{\prime}}) \text{ is a conjugate pair}\}$$

2. If $A_k$ is empty, output $u_k$ and halt

3. Locate the lexicographically largest pair $(i,i^{\prime})$ in
$A_k$

4. Form the set

$$B_k(i,i^{\prime})=\{(j,j^{\prime}): i+1\leq j^{\prime}\leq i^{\prime}<j \text{ and }  (x_j,x_{j^{\prime}}) \text{ is a conjugate pair}\}$$

5. If $B_k(i,i^{\prime})$ is empty, delete $(i,i^{\prime})$ from
$A_k$ and go to Step (2)

6.  Locate the lexicographically largest pair $(j,j^{\prime})$ in
$B_k(i,i^{\prime})$

7. Form the cross-joined sequence $u^{\prime}$ based on the
cross-join pairs $(x_i,x_{i^{\prime}})$ and $(x_j,x_{j^{\prime}})$:

$s^{\prime}=(x_1,\ldots,x_i,x_{i^{\prime}+1},\ldots,x_j,x_{j^{\prime}+1},\ldots,x_{i^{\prime}},x_{i+1},\ldots,x_{j^{\prime}},x_{j+1},\ldots,x_N)$

8. If $u^{\prime}$ occurred earlier as an output sequence, delete
$(j,j^{\prime})$ from $B_k(i,i^{\prime})$ and go to Step (5)

9. Output $u_k$

10. Let $k=k+1$

11. Let the current sequence $u_k$ be $u^{\prime}$ and go to Step
(1)

\begin{remark}
The sets $A_k$ and $B_k(i,i^{\prime})$ need not be evaluated a
priori, but they simply make the presentation easier. Also, the
choice of $(j,j^{\prime})$ as the largest is not strictly needed for
Hamiltonicity. It is included to provide a definitive way of
choosing the next neighbor. In fact, as will be seen in the proof of
correctness next, the essential part is that the cross-join neighbor
of a current state $u_k$ that shares the highest initial subsequence
with $u_k$ will be tried first. This neighbor is accepted if it is a
new output, otherwise the next largest $(i,i^{\prime})$ is tried,
noting that (for $d>2$) $i$ may remain the same as before while
$i^{\prime}$ is lowered. That is, the next lower pair in
lexicographical order is tried.
\end{remark}

\begin{proof}[Proof of Correctness of Algorithm H]

Let $u_1,\ldots,u_K$ be the total output sequence. For the sake of
getting a contradiction, suppose that $v$ is a de Bruijn sequence
that is not in the output.  Let $k_0$ be the maximal index between
$1$ and $K$ such that $u_{k_0}$ shares the longest possible initial
subsequence with $v$. Let $l$ be the length of this common largest
initial subsequence. By Lemma~\ref{L:fundamental}, there exists a de
Bruijn sequence $\hat{v}$ that is a cross-join neighbor of $u_{k_0}$
such that $\mathcal{D}(\hat{v},v)<\mathcal{D}(u_{k_0},v)$, where
$\mathcal{D}$ is as given in Definition~\ref{D:distance}. Recalling
the proof of Lemma~\ref{L:fundamental} this implies that

(a) $\hat{v}$ shares with $u_{k_0}$ the same initial subsequence of length $l$, and

(b) it shares an initial subsequence with $v$ that is longer than $l$.

By (b), it follows that $\hat{v}$ could not have appeared earlier in
the output, as this contradicts the definition of $k_0$. This in
turn implies that $u_{k_0}$ is not the last output sequence (i.e.
$k_0<K$), for otherwise it has at least one neighbor that did not
previously appear ($\hat{v}$ constitutes one such neighbor).  Thus
let $i$ be the length of the initial subsequence that is common
between $u_{k_0}$ and $u_{k_0+1}$ and let us consider two cases:

Case 1: $l\leq i$. Then $u_{k_0+1}$ shares with $v$ a common initial
subsequence of length $l$. But this contradicts the maximality of
$k_0$.

Case 2: $l>i$. Then $\hat{v}$ is a neighbor of $u_{k_0}$ that leaves
the initial subsequence of length $l$ unchanged. This says that
there must be an index $l^{\prime}$ such that the entries of
$u_{k_0}$ at $l$ and $l^{\prime}$ are conjugate a pair that leads to
a cross-join neighbor that was not the previous output. This
contradicts the choice of $(i,i^{\prime})$ as the lexicographically
largest such pair.

\noindent The two contradictions show that the output is a Hamiltonian path.
\end{proof}

Table~\ref{T:Hamil16} displays two Hamiltonian paths for the case
when $N=16$ and $d=2$. The one on the left is the result of
Algorithm H, while the one on the right uses a version of Algorithm
H that uses the lexicographically smallest $(j,j^{\prime})$. As a
first guess, we used the prefer-one sequence, see \cite{fred1982},
as a starting sequence but none of the paths, for $N=16$, come up to
be a Hamiltonian cycle. However, we do get a Hamiltonian cycle when
we apply Algorithm H with any of the sequences (2), (3), (5), (6),
(9), (10), (11), (12), where the labels are as in
Table~\ref{T:Hamil16}. By an extensive search we also find that the
de Bruijn sequence

\noindent{\small 0,  1,  3,  7, 15, 31, 30, 28, 24, 17,  2,  5, 10, 21, 11, 23, 14, 29, 26, 20,  9, 19,  6, 13, 27, 22, 12, 25, 18,  4,  8, 16}

\noindent serves as an initial sequence of Hamiltonian cycle for
$N=32$ and $d=2$.

For $d=2$ and  $N=4,6,8,10,12,14,18,20,22,24,26, 28,30$, the algorithm can easily locate
an initial de Bruijn sequence that initiates a Hamiltonian cycle. An implementation in Rstudio is given at the end of this paper.

\begin{table}
\center \tiny \noindent\begin{tabular}{l|l}

1) $\;\, 0, 1, 3, 7, 15, 14, \overline{13}, 11, 6, \underline{12},
9, 2, \overline{5}, 10, \underline{4}, 8$ & 1) $\;\, 0, 1, 3, 7, 15,
14, \overline{13}, 11, 6, 12, 9, \underline{2}, \overline{5},
\underline{10},
4, 8$\\

2) $ \;\,0, 1, 3, 7, 15, 14, 13, \overline{10}, \underline{4}, 9,
\overline{2}, 5, 11, 6,
\underline{12}, 8$ & 3) $ \;\,0, 1, 3, 7, 15, 14, 13, \overline{10}, 5, 11, 6, \underline{12}, 9, \overline{2}, \underline{4}, 8$\\

3) $ \;\,0, 1, 3, 7, 15, \overline{14}, 13, \underline{10}, 5, 11,
\overline{6}, 12, 9, \underline{2}, 4, 8$ & 2) $ \;\,0, 1, 3, 7, 15,\textbf{}
\overline{14}, 13, 10, \underline{4}, 9, 2, 5, 11, \overline{6},
\underline{12}, 8$\\

4) $ \;\,0, 1, \overline{3}, 7, 15, 14, \underline{12}, 9, 2, 5, \overline{11}, 6, 13, 10, \underline{4}, 8$
& 4) $ \;\,0, 1, \overline{3}, 7, 15, \underline{14}, 12, 9, 2, 5, \overline{11}, \underline{6}, 13, 10, 4, 8$\\

5) $ \;\,0, 1, 3, 6, 13, \overline{10}, \underline{4}, 9, \overline{2}, 5, 11, 7, 15, 14, \underline{12}, 8$
& 8) $ \;\,0, 1, 3, \overline{6}, 12, 9, 2, \underline{5}, 11, 7, 15, \overline{14}, \underline{13}, 10, 4, 8$\\

6) $ \;\,0, 1, 3, 6, \overline{13}, \underline{10}, \overline{5}, 11, 7, 15, 14, 12, 9, \underline{2}, 4, 8$
& 7) $ \;\,0, 1, 3, 6, \overline{13}, 11, 7, 15, 14, 12, 9, \underline{2}, \overline{5}, \underline{10}, 4, 8$\\

7) $ \;\,0, 1, 3, \overline{6}, \underline{13}, 11, 7, 15, \overline{14}, 12, 9, 2, \underline{5}, 10, 4, 8$
& 6) $ \;\,0, 1, 3, 6, 13,\overline{ 10}, 5, 11, 7, 15, 14, \underline{12}, 9, \overline{2}, \underline{4}, 8$\\

8) $ \;\,0, \overline{1}, 3, 6, \underline{12}, \overline{9}, 2, 5, 11, 7, 15, 14, 13, 10, \underline{4}, 8$
& 5) $ \;\,0, \overline{1}, 3, 6, 13, \underline{10}, 4, \overline{9}, \underline{2}, 5, 11, 7, 15, 14, 12, 8$\\

9) $ \;\,0, 1, 2, 5, 11, 7, 15, \overline{14}, 13, 10, \underline{4}, 9, 3, \overline{6}, \underline{12}, 8$
& 15) $0, 1, 2, 4, 9, \overline{3}, \underline{6}, 13, 10, 5, \overline{11}, 7, 15, \underline{14}, 12, 8$\\

10) $0, 1, 2, 5, \overline{11}, 7, 15, 14, \underline{12}, 9, \overline{3}, 6, 13, 10, \underline{4}, 8$
& 16) $0, 1, \overline{2}, 4, 9, 3, 7, 15, 14, \underline{13}, \overline{10}, \underline{5}, 11, 6, 12, 8$\\

11) $0, 1, 2, 5, 11, \overline{6}, 13, 10, \underline{4}, 9, 3, 7, 15, \overline{14}, \underline{12}, 8$
& 13) $0, 1, 2, 5, 10, 4, 9, \overline{3}, 7, 15, \underline{14}, 13, \overline{11}, \underline{6}, 12, 8$\\

12) $0, 1, 2, \overline{5}, 11, 6, \underline{12}, 9, 3, 7, 15, 14, \overline{13}, 10, \underline{4}, 8$
& 14) $0, 1, 2, \overline{5}, 10, 4, 9, \underline{3}, 6, \overline{13}, \underline{11}, 7, 15, 14, 12, 8$\\

13) $0, 1, 2, 5, 10, 4, 9, \overline{3}, 7, 15, \underline{14}, 13, \overline{11}, \underline{6}, 12, 8$
& 11) $0, 1, 2, 5, 11, \overline{6}, 13, 10, \underline{4}, 9, 3, 7, 15, \overline{14}, \underline{12}, 8$\\

14) $0, 1, \overline{2}, \underline{5}, \overline{10}, 4, 9, 3, 6, \underline{13}, 11, 7, 15, 14, 12, 8$
& 12) $0, 1, 2, 5, \overline{11}, \underline{6}, 12, 9, \overline{3}, 7, 15, \underline{14}, 13, 10, 4, 8$\\

15) $0, 1, 2, 4, 9, \overline{3}, \underline{6}, 13, 10, 5, \overline{11}, 7, 15, \underline{14}, 12, 8$
& 10) $0, 1, 2, 5, 11, 7, 15, \overline{14}, \underline{12}, 9, 3, \overline{6}, 13, 10, \underline{4}, 8$\\

16) $0, 1, 2, 4, 9, 3, 7, 15, 14, 13, 10, 5, 11, 6, 12, 8$ & 9) $
\;\,0, \overline{1}, 2, 5, \underline{11}, 7, 15, 14, 13, 10, 4,
\overline{9}, \underline{3}, 6, 12, 8$\\\hline
\end{tabular}
\caption{$i$ and $i^{\prime}$ are overlined while $j$ and
$j^{\prime}$ are underlined. The labels on the right refer to the
sequences on the left.}
\end{table}\label{T:Hamil16}

\section{Conclusion}\label{S:conclusion}
When $d$ divides $N$, we have proven that the cross-join graph that
corresponds to the generalized de Bruijn digraph $G_B(N,d)$ with an
arbitrary number of vertices $N$ is connected, and indeed
Hamiltonian, in the sense that it contains Hamiltonian paths. This
of course includes the class of all regular de Bruijn sequences of
arbitrary alphabet size $d$. Usefulness of this result includes in
the fact that a linearly generated sequence can be cross-joined a
number of times to obtain virtually any nonlinear sequence. Even
though the Hamiltonian path can be altered into a Hamiltonian cycle
in all cases considered in computation, the existence of these
cycles is formally still open.  Of particular interest is a
condition for an initial de~Bruijn sequence that makes Algorithm H a
Hamiltonian cycle.

As the proof of the fundamental Lemma~\ref{L:fundamental} depends on
the fact that the conjugacy of vertices of $G_B(N,d)$ is an
equivalence relation, the results of this paper are still open when
$d$ does not divide $N$.

\appendix
\section{Implementation in Rstudio}

\begin{verbatim}
largest_i<-function(i,dB,d){
  #input: a de Bruijn sequence dB and a pair of locations i[1]<i[2]
  #output: The largest conjugate pair (i[1],i2]) with i[1]<i[2] that
  #is less than or equal, lexicographically, to the input i
  N=length(dB);
  i1=i[1];
  i2=i[2];
  M=N/d;
  while(i1>0){
    while(i2>i1){
      if(dB[i1]%%M==dB[i2]%%M) break
      else i2=i2-1
    }
    if(i2==i1){
      i1=i1-1;
      i2=N-1
    }
    else break #avoid infinite loop
  }
  return(c(i1,i2))
}

largest_j<-function(i,j, dB,d){
  #INPUT: a de Bruijn sequence dB, a conjugate pair (i[1],i[2])
  # in dB with i[1]<i[2], and another pair of locations (j[1],j[2])
  # with  j[1]>i[2], i[1]<j[2]<=i[2]
  #OUTPUT: The largest conjugate pair (j[1],j[2]) with j[1]>i[2]
  #and i[1]+1<j[2]<=i[2] that is less than or equal, lexicographically,
  #to the input j and that makes a cross-join pair with (i[1],i[2])
  #it returns j[1]=i[2] if no such pair j is possible for pair i
  N=length(dB);
  j1=j[1];
  j2=j[2];
  M=N/d;
  while(j1>i[2]){
    while(j2>i[1]){
      if(dB[j1]%%M==dB[j2]%%M) break #found a conj. pair at j1, j2
      else j2=j2-1
    }
    if(j2==i[1]){
      j1=j1-1;
      j2=i[2] #begin with (j1,j2) as lex-largest with new reduction in j1
    }
    else break #avoid infinite loop: continue inner break
  }
  return(c(j1,j2))
}

crossjoin_HP<-function(dB,d=2){
  #INPUT: generalized deBruijn cycle dB that starts with 0, d is multiplier
  #OUTPUT: a list X that includes a Hamiltonian path across
  #the cross-join graph starting with dB as the first vertex
  X=list(dB);
  N=length(dB);
  L=1; #will store current length of Hamil path
  i=c(N-2,N-1);
  j=c(N,N-1);
  while(i[1]>0){
    i=largest_i(i,dB,d); #get lex largest conjugate pair i
                         #that is less or equal to input i
    j=c(N,i[2]);
    if (i[1]<=0) break;
    while(j[1]>i[2]){
      j=largest_j(i,j,dB,d);
      if((j[1]>i[2]) && (j[2]>i[1])){#just located cross-join partner j for i
        B=c(dB[1:(i[1])],dB[(i[2]+1):(j[1])]);
        if ((i[2]-i[1])==1) B=c(B,dB[i[2]]) else{
          if (j[2]<i[2]) B=c(B,dB[(j[2]+1):(i[2])]);# else B=c(B,dB[i[2]]);
          if (i[1]<j[2]) B=c(B,dB[(i[1]+1):(j[2])]);}# else B=c(B,dB[j[2]]);}
        if (j[1]<N) B=c(B,dB[(j[1]+1):N]);
        if (is_new(B,X)){
          #L=L+1;
          L=length(X)+1;
          X[[L]]=B;
          dB=B;
          #reset i and j for new iteration
          i=c(N-2,N-1); j=c(N,N-1);
          break
        }
        else {#adjust j
          if (j[2]>i[1]+1) j[2]=j[2]-1
          else {j[1]=j[1]-1;j[2]=i[2]}
        }
      }
    }
    if(i[2]>i[1]+1) i[2]=i[2]-1
    else {i[1]=i[1]-1;i[2]=N-1}
  }
  return(X)
}
\end{verbatim}

\end{document}